\RequirePackage{tikz}
\documentclass[default,iicol]{sn-jnl}

\usepackage{amsfonts}
\usepackage{amsmath}
\usepackage{amssymb}
\usepackage{mathtools}
\usepackage{xcolor}

\jyear{2025}

\theoremstyle{thmstyleone}
\newtheorem{theorem}{Theorem}
\newtheorem{proposition}[theorem]{Proposition}

\theoremstyle{thmstyletwo}

\theoremstyle{thmstylethree}
\newtheorem{definition}{Definition}
\newtheorem{lemma}{Lemma}

\DeclareMathOperator*{\argmin}{arg\ min}
\newcommand{\del}{\, \mathrm{d}}

\newcommand{\kdis}{\mathtt{K}}
\newcommand{\odis}{\mathtt{f}}
\newcommand{\ndat}{g^{\,\delta}}
\newcommand{\ndatdis}{\texttt{g}^{\,\delta}}

\global\long\def\Leins{L^{1}\left(\Omega,\mathbb{R}^+_0\right)}
\newcommand{\norm}[2][\infty]{ \left\Vert #2 \right\Vert_{#1} }
\global\long\def\d{\,\mathrm{d}}
\renewcommand{\R}{\mathbb{R}}
\renewcommand{\S}{\mathbb{S}^2}
\renewcommand{\P}{\mathcal{P}}

\begin{document}

\title[Single Pixel X-Ray Transform Reconstructions]{Reconstructions of Single Pixel X-Ray Transforms with Applications in Nuclear-Disarmament Verification}

\author*[1]{\fnm{Christopher} \sur{Fichtlscherer}}\email{fcp@mit.edu}

\author[1]{\fnm{R. Scott} \sur{Kemp}}

\author[2]{\fnm{Christina} \sur{Brandt}}

\affil[1]{\orgdiv{Department of Nuclear Science and Engineering}, \orgname{Massachusetts Institute of Technology}, \orgaddress{\street{77 Massachusetts Avenue}, \city{Cambridge}, \postcode{02139}, \state{Massachusetts}, \country{United States}}}

\affil[2]{\orgdiv{Institute of Mathematics and Computer Science}, \orgname{Universität Greifswald}, \orgaddress{\street{Walther-Rathenau-Str. 47}, \city{Greifswald}, \postcode{17489}, \state{Mecklenburg–Western Pomerania}, \country{Germany}}}

\abstract{In nuclear arms control and disarmament processes, it is crucial to determine
whether an object is a nuclear weapon or not without revealing sensitive
information about it.  At the \textit{MIT: Laboratory for Nuclear Security and
Policy}, such a nuclear verification method was developed, showcasing a
transmission-based approach \cite{kemp2016physical}. This method’s essential
part rests on a mathematical operation, the Single-Pixel X-Ray Transform: a
cone of X-rays transmits an object and the remaining intensity is measured with
a single-pixel detector.  This transformation and the recovery of objects from
dimensionless single-pixel measurements more generally has only been analyzed
to a limited extent.
In this work, we investigate some of the Single Pixel X-Ray Transform's
mathematical properties. More specifically, we show that the Single Pixel X-ray
transform is non-linear, continuous, Fr\'{e}chet-differentiable and
convex.  We also introduce a method of reconstructing an object based only on a
finite number of dimensionless, noisy Single Pixel X-Ray Transform measurement
values. This method is based on Douglas-Rachford splitting and uses total
variation denoising. We present an implementation for this method, focusing on
rotational symmetric objects, as they allow the use of a one-dimensional direct
total variation denoising algorithm \cite{condat2013direct}.\thispagestyle{empty}}


\maketitle
\section{The Physical Cryptographic Verification Concept}
\label{sec:introduction}
With the expiration of the New START treaty, there is currently no nuclear arms control agreement in force globally and the number of nuclear weapons is rising again after decades of reductions.
At the beginning of 2025, nine countries possessed roughly 12,100 nuclear warheads, of which
about 9,600 are held in military stockpiles \cite{status_nuclear}.
This makes progress on nuclear arms control and disarmament all the more urgent.
Although nuclear disarmament is a general interest, there is
disagreement on how to carry it out in practical terms.
However, there is widespread consensus that such a process must be
monitored under transparent and trustworthy conditions --- \textit{verification}.
The process should be monitored on-site by a team of inspectors.
One part in the process of verifying dismantlement of a nuclear warhead is its
authentication.  The task of warhead authentication is to identify an object
that is a nuclear weapon as a nuclear weapon and to reject objects that are not. Inspectors must be sure that the host
state did not dismantle a mock object instead of the actual warhead
and kept parts of its arsenal at a secret location.  What
makes this task difficult is that the inspectors must not be given
any sensitive information about the nuclear weapon like the isotopic
composition or amount of the contained materials. This
information is considered a military secret, and it is widely argued that its disclosure is prohibited by
articles I and II of the \textit{Treaty on the Non-Proliferation of Nuclear
Weapons} \cite{npt}.  Additionally, nuclear weapon states fear that sharing such
information could weaken their deterrence capabilities.

Nevertheless, some information about the design of warheads is known.
During the Black Sea Experiment in 1989, independent American scientists were
invited by the Soviet Union to participate in a series of experiments that
included measuring the gamma spectrum of a nuclear weapon. The publications
following the Black Sea Experiment likely provide the most precise publicly available
knowledge about the construction of a real nuclear weapon. Among other things,
researchers in this experiment published theoretical models of a
simplified nuclear warhead \cite{fetter1990}.
These models are shown as an example in Figure \ref{fig:fetter}.
\begin{figure*}[ht]
\centering
\begin{tikzpicture}
  \filldraw[color=black, fill=black!60] (0.0, 0.0) circle (2.1);
  \filldraw[color=black, fill=black!30] (0.0, 0.0) circle (2.0);
  \filldraw[color=black, fill=black!70] (0.0, 0.0) circle (1.0);
  \filldraw[color=black, fill=black!50] (0.0, 0.0) circle (0.7);
  \filldraw[color=black, fill=black!80] (0.0, 0.0) circle (0.5);
  \filldraw[color=black, fill=black!00] (0.0, 0.0) circle (0.425);

  \draw[-, line width=0.2mm] (0.0,0.46) -- (3.2,0.46);
  \draw[-, line width=0.2mm] (0.6,0.0) -- (3.2,0.0);
  \draw[-, line width=0.2mm] (0.0,-0.85) -- (3.2,-0.85);
  \draw[-, line width=0.2mm] (0.0,-1.5) -- (3.2,-1.55);
  \draw[-, line width=0.2mm] (0.0,-2.05) -- (3.2,-2.05);

  \fill[black] (0.0,0.46) circle (0.06cm);
  \fill[black] (0.6,0.0) circle (0.06cm);
  \fill[black] (0.0,-0.85) circle (0.06cm);
  \fill[black] (0.0,-1.5) circle (0.06cm);
  \fill[black] (0.0,-2.05) circle (0.06cm);

  \node[] at (7.2, 0.66) {weapon-grade plutonium / uranium (4.25\,-\,5.0\,cm / 5.77\,-\,7.0\,cm)};
  \node[] at (5.6, 0.0) {beryllium reflector (2.0\,cm)};
  \node[] at (6.35, -0.85) {tamper (tungsten / uranium) (3.0\,cm)};
  \node[] at (5.3, -1.55) {high explosive (10.0\,cm)};
  \node[] at (5.3, -2.05) {aluminium case (1.0\,cm)};

\end{tikzpicture}

  \caption{Theoretical models of a nuclear weapon as presented in \cite{fetter1990}.}
\label{fig:fetter}
\end{figure*}
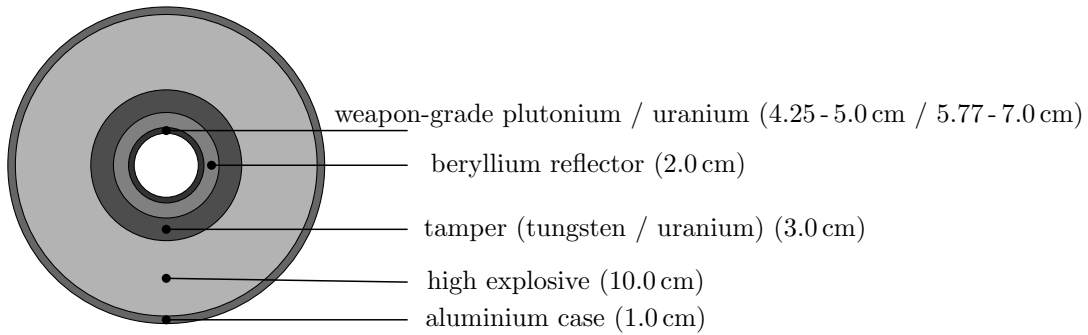
In recent years, physical proof methods have become attractive for the
task of nuclear verification \cite{glaser2014zero}.  The \textit{Physical
Cryptographic Verification Concept} (Figure \ref{fig:mit_setup}) developed at
the \textit{MIT: Laboratory for Nuclear Security and Policy} is based on
transmission nuclear resonance fluorescence \cite{kemp2016physical}.  An
electron beam generates a continuous gamma spectrum (bremsstrahlung).  This
X-ray beam of $2 \, \textrm{-} \, 9$ MeV is sent through the object.
In the object, a fraction of the photons are absorbed, which are
precisely the energy of an allowed transition of the nucleus of an existing
isotope.  This results in narrow absorption lines in the continuous spectrum.
Afterward, this spectrum with narrow lines falls on an
\textit{encrypting foil}, whose composition is only known to the host state ---
this foil contains isotopes which are also present in the nuclear weapon in
quantities unknown to the inspectors.  Nuclei in the foil
are excited by the incoming photons, which causes them to emit photons by nuclear
resonance fluorescence.
The resulting photons are then measured by a High
Purity Germanium (HPGe) detector consisting of only a single pixel.  The single pixel detector is both a practical limitation (HPGe detectors are large and expensive) and central to the information security scheme of the system. Unlike other schemes that aim to form an image of the warhead through an array of detectors, the single-pixel measurement prevents the inspector from directly observing the warhead's geometry. Imaging methods often attempt this by using a mask, but masks are not robust against misalignment or diffraction that might occur along the optical train of the system in a real experimental setup.
In the MIT scheme, this problem is avoided because each measurement returns only a single scalar value.
The measurement is performed for the template and the \textit{treaty
accountable item} (TAI), whereby the inspectors may determine the objects'
relative orientations.  If for several randomly-oriented measurements the measured values are the same for both objects within a specified error tolerance, it is assumed they are identical and the test object is declared to be as valid as the template.
In principle, the inspectors can query the measurement for any orientation of the object relative to the source--detector axis. The object is enclosed in a black box that can be repositioned and rotated freely, so measurements from source positions covering the full sphere around the object are accessible.
\begin{figure*}[ht]
  \centering
\begin{tikzpicture}[>=latex, font=\small]

  \fill[orange!35, opacity=0.5] (1.6,0.05) -- (8.4,0.7) -- (8.4,-0.7) -- (1.6,-0.05) -- cycle;
  \draw[orange!80!black, line width=0.5pt] (1.6,0.05) -- (8.4,0.7);
  \draw[orange!80!black, line width=0.5pt] (1.6,-0.05) -- (8.4,-0.7);

  \draw[green!60!black, line width=2pt, ->] (-0.3,0) -- (1.3,0);

  \fill[gray!70] (1.3,-0.55) rectangle (1.6,0.55);
  \draw[black] (1.3,-0.55) rectangle (1.6,0.55);

  \draw[gray!50, fill=gray!15, fill opacity=0.4, line width=1.5pt] (3.0,-1.4) rectangle (6.8,1.4);

  \fill[gray!40] (4.2,-0.7) rectangle (5.6,0.7);
  \draw[black] (4.2,-0.7) rectangle (5.6,0.7);
  \fill[gray!60] (4.9,0) circle (0.45);
  \fill[gray!30] (4.9,0) circle (0.33);
  \fill[gray!70] (4.9,0) circle (0.17);

  \fill[blue!30] (8.4,-0.85) rectangle (8.65,0.85);
  \draw[black] (8.4,-0.85) rectangle (8.65,0.85);

  \draw[red, line width=1.5pt, dashed] (8.65,0.35) -- (10.6,0.45);
  \draw[red, line width=1.5pt, dashed] (8.65,0.0) -- (10.6,0.0);
  \draw[red, line width=1.5pt, dashed] (8.65,-0.35) -- (10.6,-0.45);

  \fill[gray!80] (10.6,-0.85) rectangle (10.9,0.85);
  \draw[black] (10.6,-0.85) rectangle (10.9,0.85);

  \draw[red!60!black, line width=1.5pt, ->] (10.9,0.0) -- (12.2,0.0);

  \fill[gray!50] (12.2,-0.6) rectangle (13.3,0.6);
  \draw[black, line width=1pt] (12.2,-0.6) rectangle (13.3,0.6);
  \fill[blue!20] (12.4,-0.4) rectangle (13.1,0.4);
  \draw[black] (12.4,-0.4) rectangle (13.1,0.4);

  \draw[black, line width=1pt, ->] (13.3,0.0) -- (14.2,0.0);

  \node[below, align=center, font=\footnotesize\bfseries] at (0.5,-0.25) {Electron\\[-2pt]beam};

  \node[above, align=center, font=\footnotesize\bfseries] at (1.45,0.6) {Bremsstrahlung\\[-2pt]radiator};

  \node[below, align=center, font=\footnotesize\bfseries] at (2.2,-0.55) {X-ray\\[-2pt]beam};

  \node[below, font=\footnotesize\itshape] at (4.9,-1.5) {Shielded room};

  \node[above, font=\footnotesize\bfseries] at (4.9,0.75) {TAI};

  \node[below, align=center, font=\footnotesize\bfseries] at (8.5,-1.1) {Encrypting foil};

  \node[above, align=center, font=\footnotesize\bfseries] at (10.75,0.9) {Tungsten\\[-2pt]low-energy filter};

  \node[above, align=center, font=\footnotesize\bfseries] at (12.75,0.65) {HPGe\\[-2pt]detector};

  \node[right, align=center, font=\footnotesize\itshape] at (14.2,0.0) {single\\[-2pt]value};

\end{tikzpicture}
  \caption{A schematic setup with which the measurements are made according
            to the physical cryptographic method \cite{kemp2016physical}.
            During the procedure, an \textit{Electron beam} is directed through a \textit{Bremsstrahlung radiator}, generating a continuous \textit{{X}-ray beam}.
            This {X}-ray beam subsequently transmits through the \textit{treaty accountable item (TAI)}.
            As it does so, a portion of the radiation is absorbed, leading to the formation of isotope-specific absorption lines.
            To ensure the data's confidentiality, resonance fluorescence photons are produced using a specialized \textit{Encrypting foil}.
            Any low-energy photons that might carry sensitive information are then filtered out with a \textit{Tungsten low-energy filter}.
            The remaining radiation is captured using a single-pixel \textit{HPGe detector}, producing a single value. This value can be compared across different objects for consistency and verification.}
  \label{fig:mit_setup}
\end{figure*}
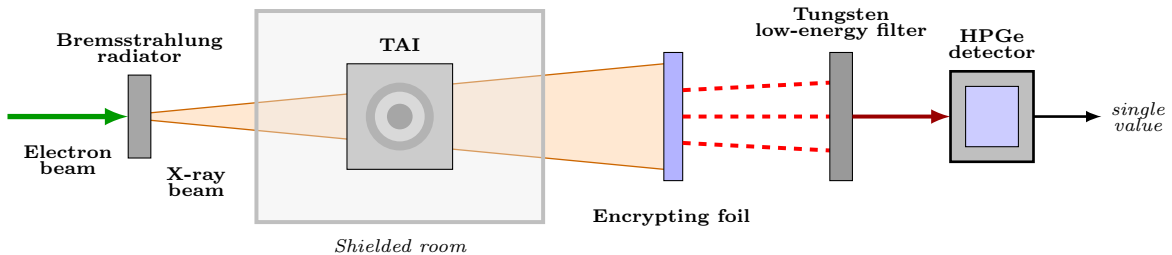
The underlying mathematical model that enables geometric comparisons from single-pixel scalar measurements is the Single Pixel X-Ray Transform (also called K-transform $\mathcal{K}$) introduced by R.~S.~Kemp and R.~R.~Macdonald \cite{kemp2016physical}.
In \cite{vavrek2016monte} Monte-Carlo simulations were
performed on the Physical Cryptographic Verification Concept. The
publication \cite{vavrek2018experimental} describes an experimental
proof-of-concept of the method. Several other verification schemes (\cite{hecla2018nuclear}, \cite{engel2019}) have also adopted the Single Pixel
X-ray transform. The more general paper \cite{gilbert2017single} discusses
the method of single-pixel imaging in radiographic inspections. It analyses the
relationship between information security and probability that inspectors can
trust the result of single-pixel imaging on the basis of several test
patterns. To the best of our knowledge, besides \cite{lai2021single}, there has been no work published dealing with the mathematical foundations or the reconstruction of the Single Pixel X-Ray Transform. In \cite{lai2021single}, stability estimates for the K-Transform are derived, a reconstruction formula based on the linearization at $f=0$ is established, uniqueness is proven both in the Euclidean and Riemannian setting, and a non-regularized numerical reconstruction is computed using a Gauss-Newton method. In contrast, our work introduces a fan-beam generalization of the K-Transform, proves additional properties such as convexity and continuous Fr\'{e}chet differentiability, and addresses the fully nonlinear regularized inverse problem using Douglas-Rachford splitting with total variation regularization. This approach provides more robust reconstructions under higher noise levels and for objects with larger densities, at the cost of increased computational effort.

Several works deal with the more general problem of single-pixel
applications, such as ghost imaging using a helicity-dependent hologram
\cite{liu2017}, 3D computational imaging \cite{sun2013}, and single-pixel
imaging with sub-Nyquist sampling \cite{Yu2020}.
The article \cite{edgar2019} gives an overview, explains the underlying
functionality, and lists some applications.
In \cite{burger2019} a more mathematical introduction and discussion of different
reconstruction methods are given. Analogous to our approach, the nonlinear reconstruction problem was
solved with the help of total-variation (TV) regularization.
However, note that, in these other approaches, single-pixel imaging is based on a combination of coded masks to achieve the desired
resolution with the help of compressed sensing techniques. This is contrary to
the applications discussed in this paper, where the single-detector is used without any
mask to avoid any spatial resolution.

This first section introduced the Physical
Cryptographic Verification Concept designed to authenticate warheads without
disclosure of sensitive information.  The operation of the Single Pixel X-ray
transform is an essential part of this concept. In the next section, we present
the mathematical model of the Single Pixel X-Ray Transform operator.  The third
section investigates the properties of this operator. Object reconstruction
from Single Pixel X-Ray Transform measurements and the discretization is
described in the fourth and fifth section, respectively. Finally, section six
shows numerical results of the reconstruction. We end with a conclusion of our
results and an outlook on open questions.
\section{Mathematical Model of the Cryptographic Verification Concept}
\label{sec:k-transform_definition}
The Single Pixel X-Ray Transform $\mathcal{K}$ does not describe the entire cryptographic verification concept
but only a part of it:
A gamma source radiates a cone of mono-energetic X-rays of known intensity $I_{0}$ \label{term:i0}
through an object.
In the object, a part of the radiation is absorbed, with
different materials absorbing to different degrees.
A \textit{single-pixel detector} \label{term:spd} measures the remaining part of the radiation.
This detector works like a single-pixel camera that receives all incoming photons in a
specific region and does not consider any spatial information.
This measurement thus results in a single measured value of $I_{1}$.
Figure \ref{fig:ktra_schem} illustrates this schematically.
Note that the encrypting foil, the filter, and the entire nuclear resonance
fluorescence absorption concept are not considered because they are not part of
the Single Pixel X-Ray Transform.
The Single Pixel X-Ray Transform is the quotient of measured $I_{1}$ and outgoing $I_{0}$ intensity,
\begin{align}
  \dfrac{I_{1}}{I_{0}} = \mathcal{K}f(r),
\end{align}
and maps by this any object $f$ to a value in $[0,1]$.
The closer this value is to 1, the more of the original intensity is captured
on the single-pixel detector. The smaller the value, the more intensity was
absorbed in the object or did not reach the detector.
The Single-Pixel X-ray Transform can be seen as a cone-beam transform using a single-pixel detector.
While the cone-beam transform maps an object to a set of individual measurements (X-ray transforms along all lines in a cone), the Single-Pixel X-ray Transform maps to the average of all these absorption quotients.

This dimensional reduction from three to one dimension makes it possible to compare single values of two objects, with no possibility to reconstruct the geometry from only a very small measurement set. Furthermore, the non-linearity of the Single-Pixel X-ray Transform $\mathcal{K}$ ensures that no conclusions about the total absorption value, and thus the total mass of the object, can be drawn from a single value.

We note that the mono-energetic model based on the Beer-Lambert law (Equation~\ref{eq:beer_2}) is a simplification. Real nuclear materials exhibit energy-dependent absorption coefficients, and the measurement system described above uses a continuous bremsstrahlung spectrum. However, the mathematical framework developed in the following sections applies to any positive absorption function. A generalization accounting for energy-dependent absorption is discussed in Section~\ref{sec:discussion_and_outlook}.

Inspired by the theoretical nuclear-weapon model in \cite{fetter1990}, we will assume that the object is rotationally symmetric, thus the support is a ball $\Omega=B_{1}(0) \subset \mathbb{R}^{3}$.  For simplicity, the ball is
centered at the origin and the radius is one.
\begin{definition}
  An object $f \in L^{1}(\Omega, \mathbb{R}^{+}_{0})$ is a rotational symmetric mapping from $\Omega=B_{1}(0) \subset \mathbb{R}^{3}$
  to its density, which is a positive real number.
  \label{term:omega}
  \label{def:object}
\end{definition}
If a beam of X-rays passes through the object, it loses intensity according to the Beer-Lambert law
\cite{beer1852bestimmung}:
\begin{align}
  I_{1} = I_{0} \cdot \text{e}^{- \int_{s_{0}}^{s_{1}}f(s) \del s}.
  \label{eq:beer_2}
\end{align}

With this idea in mind, the X-ray transform $\mathcal{P} f(x, \vartheta)$ is the
line integral along the line though the point $x$ in direction of $\vartheta$.

Denote with $\S$ the sphere in $\R^3$. Since the object $f$ is in $L^{1}$,
$\mathcal{P} f(x, \vartheta)$ is defined for every object $f$ and $\left\{ \left(x,\vartheta\right)\in\mathbb{R}^3\times \S\right\}$.
We thus have the linear bounded X-ray operator
\begin{align}
  \mathcal{P}&: L^{1}(\mathbb{R}^3) \rightarrow L^1(T)\,, \\
  \mathcal{P}f(x,\vartheta) &\coloneqq \int_{\R}  f(x+t \vartheta ) \del t\,
\end{align}
with $T:=\left\{ \left(x,\vartheta\right)\in\vartheta^{\bot}\times \S\right\} $.
Here, $\vartheta^{\bot}=\left\{ x\in\mathbb{R}^{3},x^{T}\vartheta=0\right\} $ and the orthogonal projection onto $\vartheta^{\bot}$ is denoted by $E_{\vartheta}:\mathbb{R}^{3}\rightarrow\vartheta^{\bot}$,   $E_{\vartheta}x=x-\left(x\cdot\vartheta\right)\vartheta$.
Note that the back projection
\begin{align}
P^{*} & :L^{\infty}\left(T\right)\rightarrow L^{\infty}\left(\mathbb{\mathbb{R}}^{3}\right)\,,\\
P^{*}g(x) & =\int_{\S}g\left(E_{\vartheta}x,\vartheta\right)\d\vartheta\,
\end{align}
is simply the adjoint operator of $\mathcal{P}$, see e.g., \cite{natterer2001}.
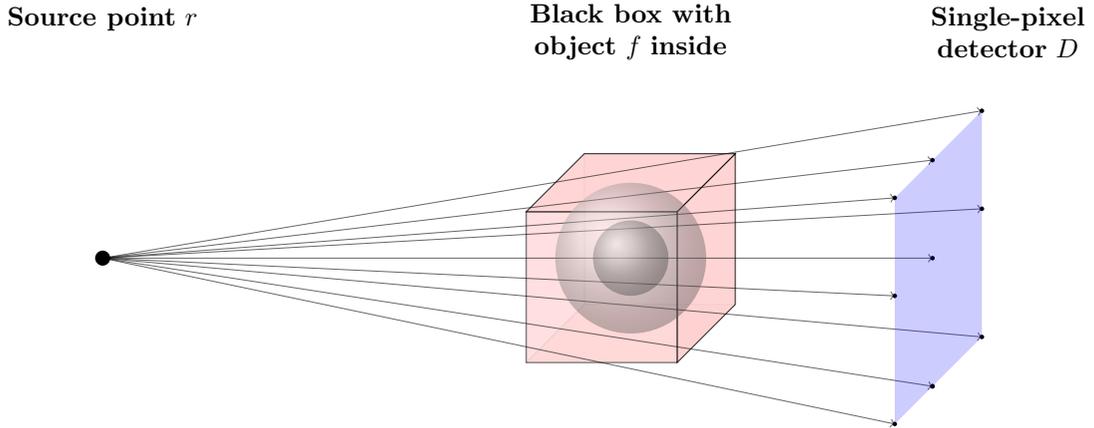
\begin{figure*}[ht]
  \centering
  \begin{tikzpicture}

\newcommand{\Dx}{-6}
\newcommand{\Dy}{1}
\newcommand{\Dz}{1}

\newcommand{\Px}{5}

\newcommand{\Pa}{-0.7}
\newcommand{\Pb}{1}
\newcommand{\Pc}{2.3}

\newcommand{\Depth}{2}
\newcommand{\Height}{2}
\newcommand{\Width}{2}

\coordinate (O) at (0,0,0);
\coordinate (A) at (0,\Width,0);
\coordinate (B) at (0,\Width,\Height);
\coordinate (C) at (0,0,\Height);
\coordinate (D) at (\Depth,0,0);
\coordinate (E) at (\Depth,\Width,0);
\coordinate (F) at (\Depth,\Width,\Height);
\coordinate (G) at (\Depth,0,\Height);

\draw[black,fill=red!80, opacity=0.1] (O) -- (C) -- (G) -- (D) -- cycle;
\draw[black,fill=red!30, opacity=0.1] (O) -- (A) -- (E) -- (D) -- cycle;
\draw[black,fill=red!10, opacity=0.1] (O) -- (A) -- (B) -- (C) -- cycle;
\draw[black,fill=red!20, opacity=0.8] (D) -- (E) -- (F) -- (G) -- cycle;
\draw[black,fill=red!20, opacity=0.6] (C) -- (B) -- (F) -- (G) -- cycle;
\draw[black,fill=red!20, opacity=0.8] (A) -- (B) -- (F) -- (E) -- cycle;

\fill[black] (\Dx, \Dy, \Dz) circle (0.1cm);
\draw[->, line width=0.1mm, opacity=0.6] (\Dx, \Dy, \Dz) -- (\Px, \Pa, \Pa);
\draw[->, line width=0.1mm, opacity=0.6] (\Dx, \Dy, \Dz) -- (\Px, \Pb, \Pa);
\draw[->, line width=0.1mm, opacity=0.6] (\Dx, \Dy, \Dz) -- (\Px, \Pc, \Pa);
\draw[->, line width=0.1mm, opacity=0.6] (\Dx, \Dy, \Dz) -- (\Px, \Pc, \Pb);
\draw[->, line width=0.1mm, opacity=0.6] (\Dx, \Dy, \Dz) -- (\Px, \Pa, \Pb);
\draw[->, line width=0.1mm, opacity=0.6] (\Dx, \Dy, \Dz) -- (\Px, \Pb, \Pb);
\draw[->, line width=0.1mm, opacity=0.6] (\Dx, \Dy, \Dz) -- (\Px, \Pc, \Pc);
\draw[->, line width=0.1mm, opacity=0.6] (\Dx, \Dy, \Dz) -- (\Px, \Pa, \Pc);
\draw[->, line width=0.1mm, opacity=0.6] (\Dx, \Dy, \Dz) -- (\Px, \Pb, \Pc);

\fill[black] (\Px, \Pa, \Pa) circle (0.03cm);
\fill[black] (\Px, \Pa, \Pb) circle (0.03cm);
\fill[black] (\Px, \Pb, \Pa) circle (0.03cm);
\fill[black] (\Px, \Pc, \Pa) circle (0.03cm);
\fill[black] (\Px, \Pc, \Pb) circle (0.03cm);
\fill[black] (\Px, \Pa, \Pb) circle (0.03cm);
\fill[black] (\Px, \Pb, \Pb) circle (0.03cm);
\fill[black] (\Px, \Pc, \Pc) circle (0.03cm);
\fill[black] (\Px, \Pa, \Pc) circle (0.03cm);
\fill[black] (\Px, \Pb, \Pc) circle (0.03cm);

  \node[align=left] at (\Dx, \Dy+3, \Dz){\textbf{Source point} $r$\\};
  \node[align=center] at (1, \Dy+3, \Dz) {\textbf{Black box with}\\\textbf{object} $f$ \textbf{inside}};
  \node[align=center] at (\Px+1, \Dy+3, \Dz) {\textbf{Single-pixel}\\ \textbf{detector} $D$};

\shade[ball color = gray!40, opacity = 0.4] (1,1,1) circle (1cm);
\shade[ball color = gray!80, opacity = 0.4] (1,1,1) circle (0.5cm);

\fill[blue,opacity=0.2] (\Px,\Pa,\Pa) -- (\Px,\Pa,\Pc) -- (\Px,\Pc,\Pc) --  (\Px,\Pc,\Pa) -- cycle;
\end{tikzpicture}
  \caption{On the left, a source $r$, which emits X-rays of intensity $I_{0}$,
          is shown.
          In this cone, the black box, containing an unknown object $f$, is placed.
          While the X-rays pass through the black box, a part of their intensity is absorbed.
          The remaining intensity $I_{1}$ is measured with a single-pixel detector $D$}
  \label{fig:ktra_schem}
\end{figure*}
In this notation, we define the \textit{fan-beam Single Pixel X-Ray Transform}.
\begin{definition}
  \label{term:ktf}
  Let $f$ be an object and $r$ be an arbitrary point in $\mathbb{R}^{3}$.
  Let $D$ be a set of points, which has the physical interpretation of the detector.
  Further, let
  $\S_{D}$ be the set of all angles, such that the lines passing through the point $r$ also
  pass through at least one point of the detector $D$.
  Then, the \textit{fan-beam Single Pixel X-Ray Transform} of the object $f$ is defined as

\begin{align}
  \mathcal{K}_{D}&: L^{1}(\Omega, \mathbb{R}^{+}_{0}) \rightarrow L^\infty(\mathbb{R}^{3},[0,1])\,, \\
  \mathcal{K}_{D}f(r) &\coloneqq \int_{\S_{D}(r)}\text{e}^{- \mathcal{P}f(r, \vartheta)} \del \vartheta\,.
\end{align}
\end{definition}
The point $r \in \mathbb{R}^{3}$ has the physical interpretation of the source, while $D$
has the physical interpretation of the detector, see Figure
\ref{fig:ktra_schem}.
The fan-beam Single Pixel X-Ray Transform is a generalization of the Single Pixel X-Ray Transform's original definition
from \cite{kemp2016physical} which is given by the following Definition \ref{term:ktr}.
\begin{definition}
  \label{term:ktr}
For every object $f$ every arbitrary point $r \in \mathbb{R}^{3}$
  the \textit{Single Pixel X-Ray Transform} of the object $f$ is defined as
\begin{align}
  \mathcal{K}&: L^{1}(\Omega, \mathbb{R}^{+}_{0}) \rightarrow L^\infty(\mathbb{R}^{3},[0,1])\,, \\
  \mathcal{K}f(r) &\coloneqq \int_{\S}\text{e}^{- \mathcal{P}f(r,\vartheta)} \del \vartheta\,.
\end{align}
\end{definition}
While in the original definition of the Single Pixel X-Ray Transform the integral is computed over all lines
passing through the point $r$, the generalization allows integrating only over the lines
in whose direction gamma rays are emitted and later measured.

The set $\S_{D}$ plays a role analogous to the Orlov condition in classical three-dimensional parallel-beam tomography \cite{orlov1975}: in that setting, the set of projection directions must cover the unit sphere such that no great circle remains unmeasured, which guarantees sufficient data for reconstruction. Similarly, $\S_{D}$ determines the angular coverage of the fan-beam Single Pixel X-Ray Transform. For the original (non-fan-beam) definition, the integration extends over the full sphere~$\S$, so the Orlov condition is trivially satisfied. In the fan-beam case, the detector geometry $D$ determines $\S_{D}$, and a larger angular coverage improves the reconstruction quality.

We assume the fan-beam transmits the entire object as in the examples
in Figure \ref{fig:ktra_schem}.
For this case, there is a linear relationship between the fan-beam Single Pixel X-Ray Transform $\mathcal{K}_{D}$
and the original Single Pixel X-Ray Transform $\mathcal{K}$ given with \label{term:usr}
\begin{align}
  \mathcal{K}_{D}f(r) = \dfrac{\vert \mathbb{S}^{2}_{D} (r) \vert}{\vert \mathbb{S}^{2} \vert} \cdot \mathcal{K}f(r)\,.
\end{align}
The exponential part in the K-transform becomes one for all lines $l$ that are
not intersecting the object $f$. This guarantees the linear relationship
between the two formulations of the K-transform. We will work with the original
definition of the Single Pixel X-Ray Transform in the following analysis. However, we
use the fan-beam Single Pixel X-Ray Transform for the computer implementation
since this accelerates the computation.

Determining for an object $f$ the value of a single-pixel measurement $\mathcal{K}(f)=g$ is the
\textbf{forward problem} of the Single Pixel X-Ray Transform.
A single-pixel measurement will be disturbed by noise $g^{\delta} = g + \delta$.
In the following, we will approximate the Poisson noise by Gaussian noise.
The ill-posed \textbf{inverse problem} related to the forward operation of
the Single Pixel X-Ray Transform is to reconstruct an object from several noisy Single Pixel X-Ray Transform measurement values.
\section{Properties of the Single Pixel X-Ray Transform}
\label{sec:k-transform_properties}
In this section, we establish some basic properties of the Single Pixel X-Ray Transform.
The non-linearity of the Single Pixel X-Ray Transform results directly from the non-linearity
of the exponential function.
\begin{proposition}
The Single Pixel X-Ray Transform is a convex, non-linear, continuous operator.
\label{prop_1}
\end{proposition}
\begin{proof}
  Let $f_{1}, f_{2} \in \Leins$ be arbitrary objects, $r \in \mathbb{R}^{3}$ and
  $\lambda \in [0,1]$.
  The non-linearity can be shown by
  \begin{align*}
    \mathcal{K}(f_{1} + f_{2})(r) &=
    \int_{\S}\text{e}^{- \mathcal{P}(f_{1}+f_{2})(r,\vartheta)} \del \vartheta  \\
    &\neq \int_{\S}\text{e}^{- \mathcal{P}f_{1}(r,\vartheta)} + \text{e}^{- \mathcal{P}f_{2}(r,\vartheta)} \del \vartheta \\
        &= \mathcal{K}f_{1}(r) + \mathcal{K}f_{2}(r).
  \end{align*}
  The convexity can be shown by making use of the linearity of the X-ray transform and   the convexity of the exponential function:

      \begin{align*}
      &\mathcal{K}(\lambda f_{1} + (1- \lambda) f_{2})(r) \\
      &= \int_{\S} \text{e}^{-(\lambda \mathcal{P}f_{1}(r,\vartheta) +
      (1 - \lambda)\mathcal{P} f_{2}(r,\vartheta))} \del \vartheta \\
       &\leq \int_{\S} \lambda \text{e}^{-\mathcal{P}f_{1}(r,\vartheta)}
      + (1 - \lambda) \text{e}^{-\mathcal{P}f_{2}(r,\vartheta)} \del \vartheta \\
      &= \lambda \mathcal{K}f_{1}(r) + (1 - \lambda)  \mathcal{K}f_{2}(r).
    \end{align*}
The continuity of $\mathcal{K}$ can be established rigorously as follows: Let $\{f_n\}$ be a sequence in $L^1(\Omega, \mathbb{R}_0^+)$ with $f_n \to f$. By the continuity of the X-ray transform $\P$, we have $\P f_n \to \P f$, which implies $\text{e}^{-\P f_n} \to \text{e}^{-\P f}$ pointwise. Since $f_n \geq 0$, we have $\lvert\text{e}^{-\P f_n}\rvert \leq 1$ for all $n$, providing a uniform bound. By the Dominated Convergence Theorem, it follows that $\mathcal{K}f_n \to \mathcal{K}f$, establishing the continuity of $\mathcal{K}$. Alternatively, continuity also follows directly from the Fr\'{e}chet differentiability established in Lemma~\ref{lem:conti}.
\end{proof}
The non-linearity of the forward operator $\mathcal{K}$ makes it much harder to solve
the inverse problem of the Single Pixel X-Ray Transform --- i.e.,
to reconstruct an object from several Single Pixel X-Ray Transform measurement values since there is no unified
theory for solving non-linear inverse problems.
The differentiability of $\mathcal{K}$ allows us to use a Newton method for solving the sub-problem involving the data term  when applying the Douglas-Rachford splitting for the computation of the regularized reconstruction.

\begin{lemma}
  The operator $\mathcal{K}$ is continuously differentiable with derivative
\begin{align*}
D  \mathcal{K}\left(f\right)h & =-\int_{\S} \text{e} ^{-\P f\left(r,\vartheta\right)}\P h\left(r,\vartheta\right)\d\vartheta \\ & =-\left(\P^{*}\text{e}^{-\P f\left(r,\cdot\right)},h\right)_{L^{\infty}\left(\mathbb{R}^{3}\right)\times L^{1}\left(\mathbb{R}^{3}\right)}.
\end{align*}
\label{lem:conti}
\end{lemma}
\begin{proof}
Define $\Delta r:= \mathcal{K}\left(f+h\right)(r)- \mathcal{K}f(r)-D \mathcal{K}(h)$, i.e.
\begin{align*}
\Delta r  & =\int_{\S}\text{e}^{-\P f\left(r,\vartheta\right)}\left[\text{e}^{-\P h\left(r,\vartheta\right)}-1+\P h\left(r,\vartheta\right)\right]\text{d} \vartheta\,.
\end{align*}
Because of
\[
\left\Vert \text{e}^{-\P h}-1+\P h\right\Vert \leq \frac{1}{2}\left\Vert \P\right\Vert ^{2}\left\Vert h\right\Vert ^{2}+\mathcal{O}\left(\left\Vert h\right\Vert ^{3}\right)\,,
\]
for $h$ tending towards zero, we obtain
\begin{align*}
& \norm[\infty]{\Delta r } \\ & \leq \int_{\S}\left\Vert  \text{e}^{-\P f\left(r,\vartheta\right)}\right\Vert_{\infty}\left\Vert \text{e}^{-\P h\left(r,\vartheta\right)}-1+\P h\left(r,\vartheta\right)\right\Vert_1 \d\vartheta \\
&  \leq\text{e}^{ \norm[1]{\P f}}
\left(\frac{1}{2}\left\Vert \P\right\Vert ^{2}\left\Vert h\right\Vert_{1}^{2}+\mathcal{O}\left(\left\Vert h\right\Vert _{1}^{3}\right)\right)    \underbrace{\int_{\S}\d\vartheta}_{=4\pi}\,
\end{align*}
 and thus
\begin{align*}
& \lim_{\norm[1]{h}\rightarrow0} \frac{\text{\ensuremath{\norm{\Delta r}} }}{\norm[1]h} \\
& \leq \lim_{\norm[1]{h}\rightarrow0} 4\pi \text{e}^{\norm[1]{\P f}}
\left(\frac{1}{2}\left\Vert \P\right\Vert ^{2}\left\Vert h\right\Vert _{1}+\mathcal{O}\left(\left\Vert h\right\Vert _{1}^{2}\right)\right) \\
&  =0\,.
\end{align*}
Here, we used the estimate
\begin{align*}
\left\Vert \text{e}^{-\P f}\right\Vert_{\infty} \leq \text{e}^{\left\Vert \P f\right\Vert_{1}}
\end{align*}
which holds since $f$ is a positive mapping (Definition~\ref{def:object}), so $\P f \geq 0$. This implies $\left\Vert \text{e}^{-\P f}\right\Vert_{\infty} \leq \text{e}^{-0} = 1$. On the other hand, $\text{e}^{\left\Vert \P f\right\Vert_{1}} \geq \text{e}^{0} = 1$, which establishes the inequality.
\end{proof}
\section{Object Reconstruction from Single Pixel X-Ray Transform Data}
\label{sec:reconstruction}
We expect sharp edges between the different material layers because of the theoretical model for a nuclear weapon.
It is well known that with the TV regularization the object's sharp edges can be maintained \cite{burger2013guide}.
For the reconstruction of the object from a set of noisy
Single Pixel X-Ray Transform values, we therefore need to solve the minimization problem
\begin{align}
  \label{eq:minimization_problem_3d}
  \min_{f \in W^{1,1}(\Omega)}
  \dfrac{1}{2} \sum_{r \in R} \big\lVert \mathcal{K} f(r) - \ndat(r) \big\rVert_{2}^{2}
  + \alpha \left\lVert \nabla f \right\rVert_{1}.
\end{align}
Here $R \subset \mathbb{R}^{3}$ denotes the set of measurement points.  The
scalar variable $\alpha$ is the regularization parameter, which indicates how
strong the TV regularization, i.e., the punishment of the edges, should be
included in the reconstruction.
The noisy Single Pixel X-Ray Transform measurement values are denoted as $\ndat$.
The variable $f$ denotes the object that we want to reconstruct.
Whereby applying the Single Pixel X-Ray Transform on a one-dimensional vector should be understood as
applying the Single Pixel X-Ray Transform on the corresponding three-dimensional rotational symmetric
object.
We define
\begin{align}
  F(f) \coloneqq \dfrac{1}{2} \sum_{r \in R} \big\lVert \mathcal{K} f(r) - \ndat(r) \big\rVert_{2}^{2}
  \hspace{15mm}
\end{align}
and
\begin{align}
  G(f) \coloneqq \alpha \left\lVert \nabla f \right\rVert_{1}.
  \label{eq:minimization_problem_1d}
\end{align}
For the minimization of the functional
\begin{align}
  J&: \Leins \rightarrow \mathbb{R}_{0}^{+}, \nonumber\\
  J(f) &= F(f) + G(f),
\end{align}
the Douglas-Rachford splitting provides an iterative way of finding the minimum
if $F$ and $G$ are maximal monotone operators.
It has been shown that the convergence also follows under the assumption that
$F$ and $G$ are proper, lower semi-continuous and convex \cite{bauschke2011convex, combettes2011proximal}.

However, we will only solve a discretized version of the problem using this splitting approach, so we first turn to the discretization of this problem in the next section.

\subsection*{Discretization of the Single Pixel X-Ray Transform}
\label{sec:k-transform_discrete}
For the implementation of the fan-beam Single Pixel X-Ray Transform measurement value $\mathcal{K}f(r)$,
we will not integrate over all angles $\S_{D}$ but only a finite subset of these angles $S^{2}_{D}$ (see Figure \ref{fig:ktra_schem}).
So we consider a discretized version of the Single Pixel X-Ray Transform $\kdis_{D}$.
\begin{definition}                                                                                      Let $ \Psi_m:L^{\infty}\left(\R^{3},\left[0,1\right]\right)\rightarrow\left[0,1\right]^{m}
$ be the observation operator mapping to the measurements $g_j=\mathcal{K} f(r_j)$ of  a finite subset of sources $r_j \in S_{D,m}\subset \S_D$, $j=1,\dots,m$. The  points are chosen in a way such that the corresponding intersection points of the line with the sphere are uniformly distributed with a fill distance $h_{X_{i} ,D}$. The discrete  Single Pixel X-Ray Transform of the object $f$ is then  defined as
\begin{align}
  \kdis_{S_{D,m}} f(r) \coloneqq  \Psi_m \circ \mathcal{K}f(r) =\int_{S_{D,m}}\text{e}^{- \mathcal{P}f(r, \vartheta)} \del \vartheta.
\end{align}
\end{definition}
\noindent
Let
\begin{align*}
  \texttt{S}_{D}(r)_{1} &\subset \dots \subset \texttt{S}_{D}(r)_{i} \subset \texttt{S}_{D}(r)_{i+1} \subset  \dots  \\ &\subset S_{D, m} \subset \S_{D}(r)
\end{align*}
be a nested sequence of the set of rays going through the point $r$.
If the fill distance $h_{X_{i} ,D}$ tends of $ \texttt{S}_{D}(r)_{i}$
tends towards zero on the section of $\S_{D}(r)$,
the value of the discretized Single Pixel X-Ray Transform approaches the value of the continuous Single Pixel X-Ray Transform
\begin{align}
  \lim\limits_{i \rightarrow \infty} \kdis_{\texttt{S}_{D}(r)_{i}}f(r) = \mathcal{K}_{D}f(r).
\end{align}
We implemented the discrete Single Pixel X-Ray Transform and distributed the rays so that the
fill distance falls off very quickly.

The finite set of measurement points is denoted by $\mathtt{R} \subset \mathbb{R}^{3}$.
Further, we will reconstruct a discretized version of the object $f$. We assume that the object can be represented in a given basis ${\varphi_i}$, $i=1,\dots,M$  of a finite-dimensional subspace $X_M\subset \Leins$. In our application, the indicator functions of the voxels are a suitable choice. The object is thus given as
\begin{align}
f=\sum_{i=1}^{M}f_{i}\varphi_{i}\,.
\end{align}
Since we assume the object $f$ is rotational symmetric, we conclude that for the discrete
object $\odis$ voxels that have the same distance to the
center take the same value.
Let $\zeta$ be the function that maps every voxel to the distance to the objects center
\begin{align}
  \zeta : L^{1}(\Omega, \mathbb{R}^{+}_{0}) \rightarrow \mathbb{R}_{0}^{+}.
\end{align}
The set of all the voxel indices with different distances to the center is then given with
\begin{align}
  I_{d} = \{i \ \vert \ i& \in 1, \dots M, \ \ \zeta (\varphi_i) \neq \zeta (\varphi_j), \nonumber \\ j& \in 1, \dots, i-1 \}.
\end{align}
From this the one-dimensional version of the problem is given by
\begin{align}
  f=\sum_{i \in I_{d}} f_{i}\varphi_{i}\,.
\end{align}
and we denote the discrete vector by $\texttt{f}=(f_i)_{i\in I_d} $.
Since the Single Pixel X-Ray Transform is convex and lower semi-continuous, this also holds for the function $F$, which is also proper.
Also, the function $G$ fulfills these properties.
With the notion of \eqref{eq:minimization_problem_1d}
the \textit{Douglas-Rachford splitting algorithm} here takes the form
\begin{align}
  \texttt{x}^{k+1} &= \text{prox}_{\gamma \texttt{G}}(\texttt{y}^{k}) = \nonumber \\
  & \argmin_{\texttt{x} \in \mathbb{R}^{3}}
  \Bigg\{
    \dfrac{1}{2} \left\lVert \texttt{y}^{k} - \texttt{x} \right\rVert_{2}^{2} + \gamma \dfrac{\alpha}{N} \sum_{i=1}^{N-1} \big\lvert \texttt{x}_{i+1}-\texttt{x}_{i} \big\rvert \Bigg\},
    \label{eq:dr_xstep}
    \\
  \texttt{f}^{k+1} &= \text{prox}_{\gamma \texttt{F}}(2 \texttt{x}^{k+1} - \texttt{y}^{k}) \nonumber \\
  &= \argmin_{\texttt{f} \in \Omega}
  \Bigg\{
    \dfrac{\gamma}{2} \sum_{r \in R}\big\lVert \kdis_{\texttt{S}_{D}(r)} \texttt{f}(r) - \ndatdis(r) \big\rVert_{2}^{2} \nonumber \\
  & \hspace{2cm} + \dfrac{1}{2} \big\lVert 2 \texttt{x}^{k+1} - \texttt{y}^{k} - \texttt{f}^{} \big\rVert_{2}^{2}
  \Bigg\},
    \label{eq:dr_ystep}
  \\
  \texttt{y}^{k+1} &= \texttt{y}^{k} + \texttt{f}^{k+1} - \texttt{x}^{k}.
\end{align}

Here, $\gamma$ defines the step size of the splitting algorithm.
To determine $\odis^{k+1}$, in Equation \eqref{eq:dr_ystep}, we use the
\textit{Limited-memory Broyden-Fletcher-Goldfarb-Shanno algorithm
with bounds} \cite{byrd1995limited, zhu1997algorithm}.
This algorithm is a quasi-Newton method that is often used for solving non-linear optimization
problems.
Further, it ensures that the boundary conditions ($\odis \in [0,1]^{N}$) are met.
Calculation of the $\odis^{k+1}$ vector is responsible for most of the computing time.
To determine $x^{k+1}$ in Equation \eqref{eq:dr_xstep}, we use the direct
algorithm for TV regularization by Lauren Condat introduced in
\cite{condat2013direct}.

All this was implemented in a Python package called \texttt{ktra}.
The three main pillars of the \texttt{ktra} package are the implementation of the
Single Pixel X-Ray Transform, the implementation of the
\textit{Direct Algorithm for 1-D Total Variation Denoising} \cite{condat2013direct} tailored to
this reconstruction problem,
and a \textit{Douglas-Rachford splitting algorithm}
for reconstructing a TV regularized, discrete model from several noisy Single Pixel X-Ray Transform measurement values.
With this package, any number of Single Pixel X-Ray Transform values of any object can easily be determined
under different rotations and translations.
Further, it includes the possibility to reconstruct an object from a specific set of noisy
Single Pixel X-Ray Transform measurement values.
In particular, this reconstruction is possible for rotationally symmetric objects.

The code is mainly based on the packages \textit{NumPy} \cite{numpy},
which allows vectorization,
and \textit{SciPy} \cite{scipy}, which provides an implementation of
the \textit{Limited-memory Broyden-Fletcher-Goldfarb-Shanno algorithm} that we use.

\section{Numerical results}
\label{sec:results}
In this section, we consider the reconstruction of three rotationally symmetric
objects from Single Pixel X-Ray Transform measurement values. The first object is a ball with
constant density, the second and third are spheres consisting of two and three shells of different densities, respectively. For the
particular case of a rotationally symmetric object – like the three examples
considered here – with a known center, the three-dimensional problem can be
reduced to a one-dimensional problem. Voxels that have the same distance to the
center of the object take the same value.  Since there is no experimental
Single Pixel X-Ray Transform measurement data available, we generate this data synthetically in the first
step: We calculate the Single Pixel X-Ray Transform measurement value from different source
points and provide these values with a noise level. As an approximation to
Poisson noise at high count rates, we use Gaussian noise here. From these measurement
results, we then reconstruct the object by the just-described procedure. In Figure
\ref{fig:results} we show a cross-section of the various reconstructions. It
should be noted that the reconstruction was performed by optimizing the
reconstruction of the three-dimensional object, not only the cross-section. In
Table \ref{table:reconstruction_parameter} we list the exact parameters. Some
of the inexactness can probably be attributed to the discretization error.

The test objects are motivated by simplified models of nuclear weapons as described by Fetter et al.\ \cite{fetter1990}: the concentric shells of different densities correspond to the layered structure of a nuclear warhead (plutonium/uranium core, beryllium reflector, tungsten tamper, high explosive, and aluminium casing, see Figure~\ref{fig:fetter}). The density values used here are normalized and do not represent actual absorption coefficients of these materials, but the shell geometry captures the essential features relevant for verification. The single sphere (constant density) serves as the simplest baseline, while the two- and three-shelled spheres model increasingly realistic warhead geometries.

\begin{figure*}[ht]
  \centering
    \includegraphics[]{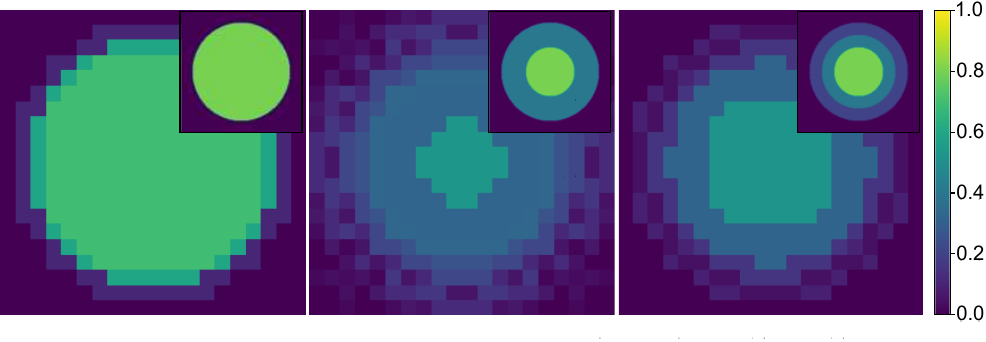}
  \caption{Results of the reconstruction of objects from Single Pixel X-Ray Transform measurement values.
  The left side always shows a cross-section through the corresponding object,
  while the right side shows the related discretized reconstruction.}
  \label{fig:results}
\end{figure*}
\begin{table*}[ht]
\centering
\begin{tabular}{r r r r}
\toprule
                                    & Sphere      & Two-shelled sphere & Three-shelled sphere \\
\midrule
Radii of Shells   & 0.8        & 0.4; 0.8 & 0.4; 0.6; 0.8          \\
Densities of the Shells        & 0.8        & 0.8; 0.4      & 0.8; 0.4; 0.2          \\
Discretization & $20 \times 20 \times 20$ & $20 \times 20 \times 20$ & $20 \times 20 \times 20$\\
  Number measurements & 1030 & 1030 & 1030 \\
Noise                               & 1\,\%        & 1\,\%               & 1\,\%          \\
Regularization Parameter ($\alpha$) &    0.1      & 0.03               & 0.05           \\
Step size ($\gamma$)                &    1        & 1                  & 1           \\
Iterations                          & 8700      & 5000               & 5000          \\
\bottomrule
\end{tabular}
    \caption{Parameters used for the reconstruction of the three test objects.}
    \label{table:reconstruction_parameter}
\end{table*}
We reconstructed the second object with the shell structure with different
noise levels and different regularization parameters in a second step. All
other parameters were chosen as in the previous reconstruction. In addition, we
created a three-dimensional discrete object of our model by dividing the object
into $20 \times 20 \times 20$ voxels, where each voxel takes the value that the function
takes at its center. We then computed the Structural Similarity Index Measure (SSIM) between the discretized
model and our reconstruction for each result. The SSIM is a perceptual metric that quantifies the similarity between two images based on luminance, contrast, and structure, yielding values in $[0,1]$ where $1$ indicates perfect agreement. The results can be
seen in Figure \ref{fig:results2}.
\begin{figure*}[!ht]
  \centering
  \includegraphics[]{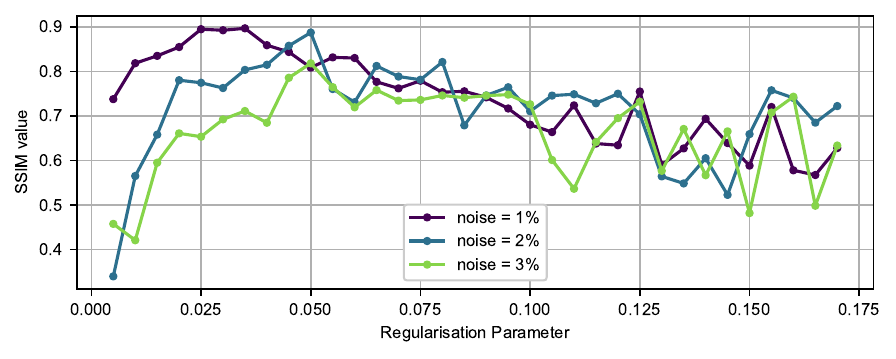}
  \caption{Analysing the dependency of the SSIM value on the noise and the regularization parameter}
  \label{fig:results2}
\end{figure*}
\section{Discussion and Outlook}
\label{sec:discussion_and_outlook}
In this publication, a generalization of the Single Pixel X-Ray Transform introduced in
\cite{kemp2016physical} was developed.
Essential properties of the Single Pixel X-Ray Transform were investigated.
We introduced a method and its implementation for solving the inverse problem of the Single Pixel X-Ray Transform under TV regularization.
We show that it is possible to reconstruct rotationally symmetric objects
with known centers from noisy Single Pixel X-Ray Transform measurements under certain limitations.

One of the great difficulties in the reconstruction process is the
discretization error.
Discretized objects can be reconstructed entirely and without further
difficulties from their noisy Single Pixel X-Ray Transform measurement values. Objects
discretized in voxel are not rotationally symmetric; in this case, the
dimension reduction step is omitted.
For a lower level of discretization, the Single Pixel X-Ray Transform values of the continuous
object and a discretized version of this object differ very much from each
other.
The higher the level of discretization, the smaller this difference becomes.
Since the reconstruction process is very computationally demanding, we can only
consider small levels of discretization in the
reconstruction. Another difficulty is that very fine reconstructions
require a tremendous number of measured values.
The reduction of the discretization error could be
reached by different methods, either a significant acceleration of
the entire reconstruction process or a statistical analysis of the error.
Another possibility is the choice of another discretization:
instead of discretizing in cube-shaped voxels, a discretization
of the rotationally symmetric object in shells could lessen the discretization error.
In order to counteract the typical effect of contrast reduction in TV regularization, two-stage reconstruction methods could be used in the future \cite{brinkmann2017bias}.

A practical rule for determining the regularization parameter $\alpha$ and the step
size $\gamma$ of the \textit{Douglas-Rachford Splitting Algorithm}
remains to be determined.

We assumed that the center of the object corresponds to the center of the black
box --- it would be of interest to allow a
displacement of the center of the object in the black box. This could be done
by introducing the center of the object as three new variables.

The Definition \ref{def:object} of an object holds only for the mono-energetic case of gamma rays. This
definition of an object could be extended by defining the object as
\begin{align}
  f: \mathbb{R}^{3} \times \mathbb{R}^{+} \rightarrow \mathbb{R}_{0}^{+},
\end{align}
which maps every energy of an incoming gamma ray at any point to an absorption
coefficient.  This generalization could model the calculation with a continuous
gamma spectrum and the nuclear resonance fluorescence absorption.

\subsection*{Code Availability}
The implementation is published under the GNU General Public License 3 \cite{gplv3} as a Python package called \texttt{ktra}, available at \url{https://github.com/cfichtlscherer/ktra} \cite{ktra}.

\end{document}